\documentclass[12pt]{amsart}
\usepackage[margin=1.1in]{geometry}
\usepackage{amsmath}
\usepackage{amssymb}
\usepackage{amsthm}
\theoremstyle{plain}
\usepackage{hyperref}
\newtheorem{theorem}{Theorem}[section]

\newtheorem*{hypothesis*}{Hypothesis}
\newtheorem{corollary}[theorem]{Corollary}
\newtheorem{lemma}[theorem]{Lemma}
\numberwithin{equation}{section}
\theoremstyle{remark}
\newtheorem{remark}[theorem]{Remark}
\title[Remarks on Landau--Siegel zeros]{Remarks on Landau--Siegel zeros}
\author{Debmalya Basak}
\address{
Debmalya Basak: Department of Mathematics,
University of Illinois, 
Altgeld Hall, 1409 W. Green Street,
Urbana, IL, 61801, USA}
\email{dbasak2@illinois.edu}
\author{Jesse Thorner}
\address{
Jesse Thorner: Department of Mathematics,
University of Illinois, 
Altgeld Hall, 1409 W. Green Street,
Urbana, IL, 61801, USA}
\email{jat9@illinois.edu}
\author{Alexandru Zaharescu}
\address{
Alexandru Zaharescu: Department of Mathematics,
University of Illinois, 
Altgeld Hall, 1409 W. Green Street,
Urbana, IL, 61801, USA and Simion Stoilow Institute of Mathematics of the Romanian Academy, 
P. O. Box 1-764, RO-014700 Bucharest, Romania}
\email{zaharesc@illinois.edu}  
\begin{document}

\dedicatory{Dedicated to Dorian Goldfeld on the 50th anniversary of his elegant proof of Siegel's theorem}

\begin{abstract}
For certain families of $L$-functions, we prove that if each $L$-function in the family has only real zeros in a fixed yet arbitrarily small neighborhood of $s=1$, then one may considerably improve upon the known results on Landau--Siegel zeros.  Sarnak and the third author proved a similar result under much more restrictive hypotheses.
\end{abstract}

\maketitle

\section{Introduction}
Building on the seminal work of Hadamard and de la Vallée Poussin for the Riemann zeta function $\zeta(s)$, it is known that if $\chi\pmod{q_\chi}$ is a primitive  Dirichlet character, then there exists an absolute and effectively computable constant $c_1>0$ such that the Dirichlet $L$-function $L(s, \chi)$ has at most one zero $\beta$ (necessarily real and simple) in the region
\[
\mathrm{Re}(s) \geq 1-c_1 / \log (q_{\chi}(|\mathrm{Im}(s)|+3)).
\]
If $\beta$ exists, then $\chi$ is real and nontrivial.  It follows from Siegel's lower bound on $L(1, \chi)$ \cite{S1935} that for any $\varepsilon>0$, there exists a constant $c_2(\varepsilon)>0$ such that if $\beta$ exists, then
\[
\beta \leq 1-c_2(\varepsilon) q_{\chi}^{-\varepsilon}.
\]
See Goldfeld \cite{G1974} for a concise proof of Siegel's lower bound on $L(1,\chi)$.  Unfortunately, no known proof provides an effective determination of $c_2$ in terms of $\varepsilon$. Define
\[
\mathcal{S}=\{\chi\pmod{q_{\chi}}\colon \textup{$\chi$ primitive and real}\}.
\]
Tatuzawa \cite{T1951} refined Siegel's result as follows:  If $ 0<\varepsilon<1/11.2$, then
\[
\#\{\textup{$\chi\in \mathcal{S}$: $q_{\chi}\geq e^{1/\varepsilon}$ and $L(s,\chi)$ has a real zero in $[1-0.655\varepsilon q_{\chi}^{-\varepsilon},1)$}\}\leq 1.
\]

Let Hypothesis H denote the hypothesis that if $\chi\in\mathcal{S}$, then all zeros of $L(s,\chi)$ lie on $\mathrm{Re}(s)=1/2$ or $\mathrm{Im}(s)=0$.  In other words, the generalized Riemann hypothesis is assumed to hold only for the non-real zeros, so Landau--Siegel zeros are permitted to exist.  Under Hypothesis H, the work of Sarnak and Zaharescu \cite[Proof of Theorem 1]{SZ2002} implies the following improvement of Tatuzawa's theorem:  For all $\varepsilon>0$, there exists an effectively computable constant $c_3(\varepsilon)>0$ such that
\begin{equation}
\label{eqn:SarnakZaharescu}
\#\{\textup{$\chi\in \mathcal{S}$: $q_{\chi}\geq c_3(\varepsilon)$ and $L(s,\chi)$ has a real zero in $[1-(\log q_{\chi})^{-\varepsilon},1)$}\}\leq 1.	
\end{equation}
This ``exponentiates'' the quality of the zero free region at the cost of a strong hypothesis for the non-real zeros of Dirichlet $L$-functions.

Here, we will use Tur{\'a}n's power sum method to prove that \eqref{eqn:SarnakZaharescu} holds under a much weaker hypothesis.  In order to state our hypothesis, we fix $0 < \delta < 1/10$.

\begin{hypothesis*}[$H_{\delta}$]
\textup{If $\chi\in \mathcal{S}$, then all the zeros of $L(s,\chi)$ in the disk $|z-1| < \delta$ are real.}
\end{hypothesis*}

\begin{theorem}
\label{thm:main}
Fix $0 < \delta < 1/10$.  If $H_{\delta}$ is true, then for all $\varepsilon>0$, there exists an effectively computable constant $q_0=q_0(\delta, \varepsilon)>0$ such that
\[
\#\{\textup{$\chi\in \mathcal{S}$: $q_{\chi}\geq q_0$ and $L(s,\chi)$ has a real zero in $[1-(\log q_{\chi})^{-\varepsilon},1)$}\}\leq 1.
\]
\end{theorem}
\begin{remark}
Our proof provides an explicit permissible expression for $q_0$.  See \eqref{eqn:q0_def}.
\end{remark}
\begin{remark}
It follows from Heath-Brown's zero density estimate in \cite[Theorem 3]{HB} that if the $\chi\in\mathcal{S}$ are ordered by conductor $q_{\chi}$, then a density 1 subset of $\chi\in\mathcal{S}$ satisfy $H_{\delta}$.  In contrast, Hypothesis H in \cite{SZ2002} has not been verified for any nontrivial $\chi\in\mathcal{S}$ yet.
\end{remark}

Our next result, which is ineffective, is an immediate corollary of Theorem \ref{thm:main}.  

\begin{corollary}
Fix $0 < \delta < 1/10$, and assume that $H_{\delta}$ is true. For all $\varepsilon>0$, there exists an ineffective constant $c(\delta, \varepsilon)>0$ such that if $\chi\in \mathcal{S}$ and $\sigma\geq 1-c(\delta,\varepsilon)(\log q_{\chi})^{-\varepsilon}$, then $L(\sigma,\chi)\neq 0$.
\end{corollary}

The difference between our proof of Theorem \ref{thm:main} and the proof of \cite[Theorem 1]{SZ2002} is subtle.  We contrast our work with \cite{SZ2002} in Remark \ref{rem:SZ} below.

Variants of the hypothesis $H_{\delta}$ for other families of $L$-functions lead to results for those families similar to Theorem \ref{thm:main}.  After proving Theorem \ref{thm:main}, we will sketch some examples.

\section{Preliminaries}

Fix $0 < \delta < 1/10$.  It suffices to let $0<\varepsilon<1$.  We define
\begin{align}
\label{eqn:q0_def}
q_0 = \exp\Big(\exp\Big(\frac{10\,000}{\delta^3 \varepsilon^2}\Big)\Big).
\end{align}
Let $q_1$ and $q_2$ be positive integers such that
\begin{equation}
\label{eqn:q_inequality}
q_0\leq q_2\leq q_1.
\end{equation}
For $j\in\{1,2\}$, let $\chi_j\pmod{q_j}$ be distinct primitive real Dirichlet characters.  Let $\psi$ be the primitive Dirichlet character that induces $\chi_1 \chi_2$.  It follows that $\psi$ is real, with conductor at most $q_1^2$.  We assume $H_{\delta}$, which implies that $L(s,\chi_1)$, $L(s,\chi_2)$, and $L(s,\psi)$ have no non-real zeros in the disk $|z-1|<\delta$.  The Riemann zeta function $\zeta(s)$ provably has no zero in this disk, so it follows that
\begin{equation}
\label{eqn:D_def}
D(s)=\zeta(s)L(s,\chi_1)L(s,\chi_2)L(s,\psi)
\end{equation}
has no non-real zeros in the disk $|z-1|<\delta$.

Let $\beta_j$ be the greatest real zero of $L(s,\chi_j)$.  We suppose to the contrary that
\begin{equation}
\label{eqn:beta_j}
\beta_1\geq 1-(\log q_1)^{-\varepsilon}\qquad\textup{and}\qquad \beta_2\geq 1-(\log q_2)^{-\varepsilon}.
\end{equation}
It follows from \eqref{eqn:q0_def}, \eqref{eqn:q_inequality}, and \eqref{eqn:beta_j} that
\begin{equation}
\label{eqn:max_beta_j}
\max\{1-\beta_1,1-\beta_2\}\leq (\log q_0)^{-\varepsilon}\leq\frac{\delta}{10}.
\end{equation}
We define
\begin{equation}
\label{eqn:eta_def}
\eta = \frac{\delta}{e}+\beta_2-1.
\end{equation}
It follows from \eqref{eqn:q0_def}, \eqref{eqn:q_inequality}, and \eqref{eqn:beta_j} that
\begin{equation}
\label{eqn:eta_ineq}
\frac{\delta}{2e}\leq \eta\leq \frac{\delta}{e}
\end{equation}

\section{Proof of Theorem \ref{thm:main}}
\label{sec:proof}

We begin with a relation for $\beta_1$, $\beta_2$, and all non-real zeros of $D(s)$.

\begin{lemma}
\label{lem:1}
Let $k\geq 1$ and $\ell\geq 2$ be integers.  If $\omega$ denotes a zero of $D(s)$ and $\eta$ is as in \eqref{eqn:eta_def}, then
\begin{align*}
\label{Final Simplification}
\frac{1}{\eta^{k\ell}}-\frac{1}{(1+\eta-\beta_1)^{k\ell}} \geq \frac{1}{(1+\eta-\beta_2)^{k\ell}} + \mathrm{Re} \sum_{\mathrm{Im}(\omega) \neq 0} \frac{1}{(1+\eta-\omega)^{k\ell}}.
\end{align*}
\end{lemma}
\begin{proof}
If $\mathrm{Re}(s)>1$, then we have the Dirichlet series expansion
\[
-\frac{D'}{D}(s)=\sum_{\textup{$p$ prime}}\sum_{m=1}^{\infty}\frac{(1+\chi_1(p^m)+\chi_2(p^m)+\psi(p^m))\log p}{p^{ms}}=\sum_{n=1}^{\infty}\frac{a_D(n)}{n^s}.
\]
Since $D(s)$ is the Dedekind zeta function of a biquadratic extension of $\mathbb{Q}$, we have that $a_D(n)\geq 0$ for all $n\geq 1$.  On the other hand, $D(s)$ has a Hadamard product factorization.  In particular, there exist $a_D,b_D\in\mathbb{C}$ such that
\[
(s-1)D(s) = s^{\mathrm{ord}_{s=0}D(s)}e^{a_D+b_D s} \prod_{\substack{\omega\neq 0 \\ D(\omega)=0}}\Big(1-\frac{s}{\omega}\Big)e^{s/\omega}.
\]
We emphasize that $\omega$ ranges over the trivial and nontrivial zeros of each factor of $D(s)$.

When $\mathrm{Re}(s)>1$, we equate the Dirichlet series expansion of $-(D'/D)(s)$ with the logarithmic derivative of the Hadamard product of $D(s)$, thus obtaining
\[
\sum_{n=1}^{\infty}\frac{a_D(n)}{n^s}=\frac{1}{s-1}-b_D-\frac{\mathrm{ord}_{s=0}D(s)}{s}-\sum_{\omega\neq 0}\Big(\frac{1}{s-\omega}+\frac{1}{\omega}\Big).
\]
Let $k\geq 1$ and $\ell\geq 2$ be integers.  We take the real part of the $(k\ell-1)$-th derivative of both sides, arriving at
\[
\frac{1}{(k\ell-1)!}\mathrm{Re}\sum_{n=1}^{\infty}\frac{a_D(n)(\log n)^{k\ell-1}}{n^s}=\mathrm{Re}\Big(\frac{1}{(s-1)^{k\ell}}-\sum_{\omega}\frac{1}{(s-\omega)^{k\ell}}\Big).
\]
Now, the possible trivial zero $\omega=0$ is now included in the sum over $\omega$.

We let $s=1+\eta$.  Since $a_D(n)\geq 0$ uniformly, it follows that
\[
\mathrm{Re}\sum_{\omega}\frac{1}{(1+\eta-\omega)^{k\ell}}<\frac{1}{\eta^{k\ell}}.
\]
Note that if $\omega$ is a real zero of $D(s)$, then $\omega<1$ and $(1+\eta-\omega)^{k\ell}\geq 0$.  Since $\beta_1$ and $\beta_2$ are real zeros of $D(s)$, we conclude via nonnegativity that
\[
\frac{1}{(1+\eta-\beta_1)^{k\ell}}+\frac{1}{(1+\eta-\beta_2)^{k\ell}}+\mathrm{Re}\sum_{\substack{\omega \\ \mathrm{Im}(\omega)\neq 0}}\frac{1}{(1+\eta-\omega)^{k\ell}}\leq \mathrm{Re}\sum_{\omega}\frac{1}{(1+\eta-\omega)^{k\ell}}<\frac{1}{\eta^{k\ell}}.
\]
The desired result follows.
\end{proof}

We define a sequence of complex numbers $\{z_n\}_{n=1}^{\infty}$ as follows.  Let
\begin{align}
\label{eqn:z1_def}
z_1 = (1+\eta-\beta_2)^{-\ell}.
\end{align} 
Choose an ordering $\{\omega_j\}_{j=2}^{\infty}$ of the zeros $\omega$ with $\mathrm{Im}(\omega)\neq 0$ such that if
\[
z_j = (1+\eta-\omega_j)^{-\ell},
\]
then $|z_2|\geq |z_3|\geq\cdots\geq |z_j|\geq|z_{j+1}|\geq \cdots$.  It follows from \eqref{eqn:max_beta_j}, \eqref{eqn:eta_ineq}, and our assumption of $H_{\delta}$ that if $\omega$ is a non-real zero of $D(s)$, then
\[
|1+\eta-\omega|\geq \sqrt{\eta^2+\delta^2}\geq\eta+\frac{\delta}{2}\geq 1+\eta-\beta_2.
\]
In particular, we have that $|z_1|\geq |z_2|\geq |z_3|\geq\cdots$.

\begin{lemma}
\label{lem:zero_count}
Fix $0 < \delta < 1/10$.  Let $\eta$ be as in \eqref{eqn:eta_def}, and let $\ell = \lceil \log \log q_1 \rceil$.  If $H_{\delta}$ is true, then
\[
\sum_{j=1}^{\infty}\frac{|z_j|}{|z_1|}\leq 5.
\]
\end{lemma}
\begin{proof}
Write $\omega_j=\beta_j+i\gamma_j$.  Recall that our assumption of $H_{\delta}$ implies that $D(s)$ has no non-real zeros in the disk $|z-1|<\delta$.  Consequently, $H_{\delta}$ implies that if $j\geq 2$ and $|\gamma_j|\leq 1$, then $|z_j|\leq \delta^{-\ell}$.  Per \eqref{eqn:eta_def} and \eqref{eqn:z1_def}, we have that $|z_1|=(\delta/e)^{-\ell}$.  Therefore, if
\[
N_D(T)=\#\{\omega\colon D(\omega)=0,~|\mathrm{Im}(\omega)|\leq T,~0\leq \mathrm{Re}(\omega)\leq 1\},
\]
then the sum to be estimated is at most
\begin{equation}
\label{eqn:Kbound}
1+\frac{\delta^{-\ell}}{|z_1|}\sum_{|\gamma_j|\leq 1}1+\frac{1}{|z_1|}\int_{1}^{\infty}\frac{dN_D(t)}{t^{\ell}}\leq1+\frac{\delta^{-\ell}}{(\delta/e)^{-\ell}}N_D(1)+\frac{\ell }{(\delta/e)^{-\ell}}\int_{1}^{\infty}\frac{N_D(t)}{t^{\ell+1}}dt.	
\end{equation}
 It follows from \cite[Corollary 1.2]{MR4232231} that if $T\geq 5/7$, then
\[
\Big|N_D(T)-\frac{T}{\pi}\log\Big(q_1 q_2 q_{\psi}\Big(\frac{T}{2\pi e}\Big)^4\Big)\Big|\leq \log(q_1 q_2 q_{\psi}T^4)+28.
\]
Recall that $q_2\leq q_1$ and $q_{\psi}\leq q_1 q_2$.  Thus, if $T\geq 1$, then
\[
N_D(T)\leq \frac{T}{\pi}\log\Big(q_1^3\Big(\frac{T}{2\pi e}\Big)^4\Big)+\log(q_1^3 T^4)+28.
\]
Since $\ell\geq 3$, it follows that \eqref{eqn:Kbound} is
\[
\leq 1+e^{-\ell}\Big(28-\frac{4\log(2\pi e)}{\pi}+3\Big(1+\frac{1}{\pi}\Big)\log q_1\Big)+3(\delta/e)^{\ell}\Big(\Big(1+\frac{1}{(1-\ell^{-1})\pi}\Big)\log q_1-1\Big).
\]
The desired result now follows from our choice of $\ell$, our range of $\delta$, and the bounds \eqref{eqn:q0_def} and \eqref{eqn:q_inequality}.
\end{proof}

\begin{lemma}
\label{power sum}
Fix $0 < \delta < 1/10$.  If $H_{\delta}$ is true, then there exists an integer $1\leq k\leq 120$ such that
\[
\frac{1}{(1+\eta-\beta_2)^{k\ell}} + \mathrm{Re} \sum_{\mathrm{Im}(\omega) \neq 0} \frac{1}{(1+\eta-\omega)^{k\ell}}\geq \frac{1}{8(1+\eta-\beta_2)^{k\ell}}.
\]
\end{lemma}
\begin{proof}
Let $\{y_j\}_{j=1}^{\infty}$ be a sequence of complex numbers such that $|y_1|\geq |y_2|\geq |y_3|\geq\cdots$. Set
\[
K=\sum_{j \geq 1} \frac{|y_j|}{|y_1|}.
\]
Tur{\'a}n proved that there exists $1 \leq k \leq 24 K$ such that
\[
\mathrm{Re} \sum_{j \geq 1} y_j^k \geq \frac{1}{8} \left|y_1\right|^k
\]
(see \cite[Chapter 9, Lemma 2]{M1994}).  We apply Tur{\'a}n's result to $\{z_j\}_{j=1}^{\infty}$, bounding $K$ from above using Lemma \ref{lem:zero_count}.
\end{proof}
It follows from Lemmata \ref{lem:1} and \ref{power sum} that
\begin{align*}
\frac{1}{\eta^{k\ell}}-\frac{1}{(1+\eta-\beta_1)^{k\ell}} \geq \frac{1}{8(1+\eta-\beta_2)^{k\ell}}.
\end{align*}
Multiplying through by $\eta^{k\ell}$, and observing that $\lceil \log\log q_1\rceil\leq 2\log\log q_1$, we find that
\begin{align}\label{Rewriting}
    1-\Big(1- \frac{1-\beta_1}{1+\eta-\beta_1} \Big)^{2k \log \log q_1} \geq \frac{1}{8}\Big( 1-\frac{1-\beta_2}{1+\eta-\beta_2} \Big)^{2k \log \log q_1}.
\end{align}
Since $\beta_1,\beta_2\in(0,1)$ and $\eta>0$, we find that
\[
\frac{1-\beta_1}{1+\eta-\beta_1},\frac{1-\beta_2}{1+\eta-\beta_2}\in(0,1).
\]
\begin{lemma}[Bernoulli's inequality]
\label{lem:Taylor}
If $0<a<1$ and $b>1$, then $ab > 1-(1-a)^b$.
\end{lemma}

\begin{proof}[Proof of Theorem \ref{thm:main}]
Applying Lemma \ref{lem:Taylor} to the left-hand side of \eqref{Rewriting} with
\[
a=\frac{1-\beta_1}{1+\eta-\beta_1},\qquad b = 2k\log\log q_1,
\]
we obtain the bound
\begin{align*}
\frac{2k(1-\beta_1) \log \log q_1}{1+\eta-\beta_1} &\geq \frac{1}{8}\Big( 1-\frac{1-\beta_2}{1+\eta-\beta_2} \Big)^{2k \log \log q_1}=\frac{1}{8}(\log q_1)^{2k \log\big(1-\frac{1-\beta_2}{1+\eta-\beta_2}\big)}.
\end{align*}
(We have that $b\geq 2$ by \eqref{eqn:q0_def} and \eqref{eqn:q_inequality}.)  Dividing through by $(2k\log\log q_1)/(1+\eta-\beta_1)$, we find that
\begin{equation}
\label{Big-Oh Step}
1-\beta_1\geq \frac{1+\eta-\beta_1}{16k\log\log q_1}(\log q_1)^{2k \log\big(1-\frac{1-\beta_2}{1+\eta-\beta_2}\big)}\geq \frac{\eta}{1920 \log\log q_1}(\log q_1)^{240 \log\big(1-\frac{1-\beta_2}{1+\eta-\beta_2}\big)}.\hspace{-1mm}
\end{equation}
It follows from \eqref{eqn:q0_def}, \eqref{eqn:q_inequality}, and \eqref{eqn:eta_ineq} that
\[
\varepsilon>2\frac{\log(1920 \eta^{-1}\log\log q_1)}{\log\log q_1},
\]
so
\begin{align}
\label{eqn:epsilon_range2}
\frac{\eta}{1920 \log\log q_1}(\log q_1)^{240 \log\big(1-\frac{1-\beta_2}{1+\eta-\beta_2}\big)} \geq (\log q_1)^{240 \log\big(1-\frac{1-\beta_2}{1+\eta-\beta_2}\big)-\varepsilon/2}.
\end{align}
Therefore, by \eqref{eqn:beta_j}, \eqref{Big-Oh Step}, and \eqref{eqn:epsilon_range2}, we find that
\begin{align}
\label{theta step}
(\log q_1)^{-\varepsilon}\geq (\log q_1)^{240\log\big(1-\frac{1-\beta_2}{1+\eta-\beta_2}\big)-\varepsilon/2}.
\end{align}
Recalling the definition of $\eta$ in \eqref{eqn:eta_def}, we solve for $\beta_2$ in \eqref{theta step}, thus obtaining
\begin{equation}
\label{eqn:finalbound}
\beta_2\leq 1-\frac{\delta}{e}(1-e^{-\varepsilon/480}).	
\end{equation}
However, by \eqref{eqn:q0_def} and \eqref{eqn:q_inequality}, the bound \eqref{eqn:finalbound} contradicts \eqref{eqn:beta_j}, as desired.
\end{proof}

\begin{remark}
\label{rem:SZ}
In \cite[Proof of Theorem 1]{SZ2002}, Sarnak and Zaharescu begin with the Guinand--Weil explicit formula for $D(s)$:  If $B>0$ and
\[
\phi(x) = \Big(\frac{\sin(2\pi x)}{2\pi x}\Big)^2,\qquad \widehat{\phi}(y)=\int_{-\infty}^{\infty}\phi(x)e^{-2\pi i x y}dx=\begin{cases}
\frac{1}{2}(1-\frac{|y|}{2})&\mbox{if $|y|\leq 2$,}\\
0&\mbox{otherwise,}	
\end{cases}
\]
and $\rho=\beta+i\gamma$ runs through the nontrivial zeros of $D(s)$, then
\begin{equation}
\label{eqn:Weil}
\begin{aligned}
\sum_{\rho}\phi\Big(\frac{B}{2\pi i}\Big(\rho-\frac{1}{2}\Big)\Big)&+\frac{2}{B}\sum_{n=1}^{\infty}\frac{(1+\chi_1(n)+\chi_2(n)+\psi(n))\Lambda(n)}{\sqrt{n}}\widehat{\phi}\Big(\frac{\log n}{B}\Big)\\
&=\frac{\log(q_{\psi}q_{\chi_1}q_{\chi_2})+O(1)}{2B}+2\phi\Big(\frac{B}{4\pi i}\Big)+2\phi\Big(-\frac{B}{4\pi i}\Big).
\end{aligned}
\end{equation}
Note that $\phi(0)=1$, $\phi(y)\geq 0$ for $y\in\mathbb{R}\cup i\mathbb{R}$, and $\widehat{\phi}(y)\geq 0$ for $y\in\mathbb{R}$.  Therefore, they can discard the sum over $n$ by the nonnegativity of the Dirichlet coefficients.  Assuming that $\beta=\frac{1}{2}$ or $\gamma=0$ always, they discard the contribution from each $\rho$ except for $\rho=\beta_1$ and $\rho=\beta_2$.  The conclusion \eqref{eqn:SarnakZaharescu} now follows; otherwise, they would obtain a contradiction by choosing $B = (1+3\varepsilon/4)\log\log q_1$.  See also the discussion in \cite[Section 5]{Iwaniec}, especially the remark at the end of the section.

In our proof of Theorem \ref{thm:main}, we cannot afford to discard all of the terms in the sum over zeros.  At the same time, it is unclear how to obtain a strong lower bound on the sum over non-real zeros in \eqref{eqn:Weil}.  We circumvent this problem by taking the $(k\ell-1)$-th derivative of $-(D'/D)(s)$, expressed both as a Dirichlet series and in terms of its Hadamard factorization, and bounding the sum over zeros from below using Tur{\'a}n's power sum method.  However, the power sum method will fail us if $k\ell-1$ is larger than $O(\log\log q_1)$; we handle this using $H_{\delta}$.
\end{remark}

\section{Extensions of Theorem \ref{thm:main}}

We briefly describe a way to extend Theorem \ref{thm:main} to Dirichlet characters whose order exceeds $2$.  Let $m\geq 2$ be an integer, let  $j\in\{1,2\}$, let $q_j\geq 3$ be an integer, and let $\chi_j\pmod{q_j}$ be a primitive Dirichlet character of order dividing $m$.  Note that $L(s,\chi_1)$ and $L(s,\overline{\chi}_1)$ have the same real zeros, so we assume that $\chi_2\notin\{\chi_1,\overline{\chi}_1\}$.  Let $G=\langle \chi_1,\chi_2\rangle$ be the group of Dirichlet characters generated by $\chi_1$ and $\chi_2$, and let $G^*$ be the set of primitive Dirichlet characters that induce the characters in $G$.  The Dirichlet series $D(s) = \prod_{\chi\in G^*}L(s,\chi)$ is the Dedekind zeta function of the compositum of two cyclic extensions of degree $m$ over $\mathbb{Q}$.  In particular, the Dirichlet coefficients of $D(s)$ are nonnegative.  Also, the conductor of $D(s)$ is bounded by $(q_1 q_2)^{m^2}$.  The identity element of a group is unique, so $D(s)$ has a simple pole at $s=1$ arising from the factor of $\zeta(s)$.  Note that if $\chi_2\in\langle\chi_1\rangle-\{\chi_1,\overline{\chi}_1\}$, then $D(s)$ is the product of $L$-functions of primitive characters that induce the characters in $\{\chi_1^j\colon 0\leq j\leq m-1\}$.  In order to ensure that all zeros of $D(s)$ in the disk $|z-1|<\delta$ are real, it suffices to assume for some fixed $0<\delta<1/10$ that for all $\nu\in\mathcal{S}_{m}$, all zeros of $L(s,\nu)$ in the disk $|z-1|<\delta$ are real.  Small modifications to our proof of Theorem \ref{thm:main} yield the following result.

\begin{theorem}\label{Higher Order Characters}
Fix an integer $m \geq 2$, and let $\mathcal{S}_{m}$ be the set of primitive Dirichlet characters of order dividing $m$.  Let $\sim$ be the equivalence relation on $\mathcal{S}_m$ defined by $\chi_1\sim\chi_2$ if $\chi_2\in\{\chi_1,\overline{\chi}_1\}$, and let $[\chi]$ be the equivalence class of $\chi\in\mathcal{S}_m$ in $\mathcal{S}_m/\sim$. Fix $0 < \delta < 1/10$.  Assume that if $\nu\in\mathcal{S}_{m}$, then all zeros of $L(s,\nu)$ in the disk $|z-1|<\delta$ are real.  For all $\varepsilon>0$, there exists an effectively computable constant $q_0=q_0(\delta,\varepsilon,m)>0$ such that
\[
\#\{\textup{$[\chi]\in \mathcal{S}_m/\sim$ : $q_{\chi}\geq q_0$ and $L(s,\chi)$ has a real zero in $[1-(\log q_{\chi})^{-\varepsilon},1)$}\}\leq 1.
\]
\end{theorem}
\begin{remark}
Theorem \ref{Higher Order Characters} recovers Theorem \ref{thm:main} when $m=2$.
\end{remark}

Let $\phi(x+iy)$ be a Hecke--Maa{\ss} cusp form on the modular surface $\mathrm{SL}_2(\mathbb{Z})\backslash\mathbb{H}$ with Laplace eigenvalue $\lambda_{\phi}>0$, and let $\mathrm{Sym}^2\phi$ be its symmetric square lift.  Gelbart and Jacquet \cite{GJ1978} proved that there exists a cuspidal automorphic representation of $\mathrm{GL}_3$ whose $L$-function is $L(s,\mathrm{Sym}^2\phi)$.  In \cite[Appendix]{HoffsteinLockhart}, Goldfeld, Hoffstein, and Lieman proved that there exists an absolute and effectively computable constant $c_4>0$ such that
\[
L(\sigma,\mathrm{Sym}^2\phi)\neq 0,\qquad \sigma\geq 1-c_4/(\log\lambda_{\phi}).
\]
By modifying our proof of Theorem \ref{thm:main} in a manner similar to the work of Hoffstein and Lockhart in \cite{HoffsteinLockhart}, one can prove the following result.

\begin{theorem}\label{Adjoint}
Fix $0 < \delta < 1/10$.  Let $\mathfrak{S}$ be the set of Hecke--Maa{\ss} cusp forms on $\mathrm{SL}_2(\mathbb{Z})\backslash \mathbb{H}$.  Suppose that for all $\phi_1,\phi_2\in\mathfrak{S}$ satisfying $\phi_1\neq\phi_2$, all zeros of
\[
\zeta(s)L(s,\mathrm{Sym}^2\phi_1)L(s,\mathrm{Sym}^2\phi_2)L(s,\mathrm{Sym}^2\phi_1\times \mathrm{Sym}^2\phi_2)
\]
in the disk $|z-1|<\delta$ are real.  For all $\varepsilon>0$, there exists an effectively computable constant $\lambda_0=\lambda_0(\delta,\varepsilon)$ such that
\begin{align*}
\#\{\textup{$\phi\in\mathfrak{S}$: $\lambda_{\phi}\geq \lambda_0$ and $L(s,\mathrm{Sym}^2\phi)$ has a real zero in $[1-(\log \lambda_{\phi})^{-\varepsilon},1)$}\}\leq 1.	
\end{align*}
\end{theorem}

\section*{Acknowledgements}

J.T. is partially supported by the National Science Foundation (DMS-2401311) and the Simons Foundation (MP-TSM-00002484).


\bibliographystyle{abbrv}
\bibliography{SiegelZero}

\def\cprime{$'$}
\begin{thebibliography}{10}

\bibitem{MR4232231}
M.~A. Bennett, G.~Martin, K.~O'Bryant, and A.~Rechnitzer.
\newblock Counting zeros of {D}irichlet {$L$}-functions.
\newblock {\em Math. Comp.}, 90(329):1455--1482, 2021.

\bibitem{GJ1978}
S.~Gelbart and H.~Jacquet.
\newblock A relation between automorphic representations of {${\rm GL}(2)$} and
  {${\rm GL}(3)$}.
\newblock {\em Ann. Sci. \'{E}cole Norm. Sup. (4)}, 11(4):471--542, 1978.

\bibitem{G1974}
D.~M. Goldfeld.
\newblock A simple proof of {S}iegel's theorem.
\newblock {\em Proc. Nat. Acad. Sci. U.S.A.}, 71:1055, 1974.

\bibitem{HB}
D.~R. Heath-Brown.
\newblock A mean value estimate for real character sums.
\newblock {\em Acta Arith.}, 72(3):235--275, 1995.

\bibitem{HoffsteinLockhart}
J.~Hoffstein and P.~Lockhart.
\newblock Coefficients of {M}aass forms and the {S}iegel zero.
\newblock {\em Ann. of Math. (2)}, 140(1):161--181, 1994.
\newblock With an appendix by Dorian Goldfeld, Hoffstein and Daniel Lieman.

\bibitem{Iwaniec}
H.~Iwaniec.
\newblock Conversations on the exceptional character.
\newblock In {\em Analytic number theory}, volume 1891 of {\em Lecture Notes in
  Math.}, pages 97--132. Springer, Berlin, 2006.

\bibitem{M1994}
H.~L. Montgomery.
\newblock {\em Ten lectures on the interface between analytic number theory and
  harmonic analysis}, volume~84 of {\em CBMS Regional Conference Series in
  Mathematics}.
\newblock Published for the Conference Board of the Mathematical Sciences,
  Washington, DC; by the American Mathematical Society, Providence, RI, 1994.

\bibitem{SZ2002}
P.~Sarnak and A.~Zaharescu.
\newblock Some remarks on {L}andau--{S}iegel zeros.
\newblock {\em Duke Math. J.}, 111(3):495--507, 2002.

\bibitem{S1935}
C.~L. Siegel.
\newblock {\"U}ber die classenzahl quadratischer zahlk{\"o}rper.
\newblock {\em Acta Arith.}, 1:83--86, 1935.

\bibitem{T1951}
T.~Tatuzawa.
\newblock On a theorem of {S}iegel.
\newblock {\em Jpn. J. Math.}, 21:163--178 (1952), 1951.

\end{thebibliography}
\end{document}